\documentclass[12pt]{article}
\usepackage[a4paper]{geometry}
\usepackage{amsthm}

\usepackage{amssymb}
\usepackage{amsmath}
\usepackage{eufrak}

\newtheorem{theorem}{Theorem}

\newtheorem{corollary}[theorem]{Corollary}
\newtheorem{lemma}[theorem]{Lemma}

\newtheorem{question}{Question}
\newtheorem{proposition}[theorem]{Proposition}

\begin{document}
\def\F{{\mathbb F}}
\title{ On the passage from finite braces to pre-Lie rings}

\author{ Agata Smoktunowicz}
\date{}
\maketitle
\begin{abstract} 
  Let $p$ be a prime number. 
 We show that  there is a  one-to-one correspondence between  the set of   strongly nilpotent braces  and the   set of  nilpotent pre-Lie rings of cardinality $p^{n}$, for sufficiently large $p$. 
 Moreover, there is an injective mapping from the set of  left nilpotent pre-Lie rings into the set of left nilpotent braces of cardinality $p^{n}$ for $n+1<p$. 
 As an application, by using well known results about the correspondence between braces and Hopf-Galois extensions we use pre-Lie algebras to describe Hopf-Galois extensions.
 \end{abstract}
\section{Introduction} The   Lazard correspondence is one between  Lazard $p$-Lie rings and  Lazard Lie $p$-groups, for every prime number $p$. 
In \cite{Rump}, page 141, it was  suggested that this correspondence could be extended to a correspondence  between  pre-Lie algebras  and braces of cardinality $p^{n},$  and  a method for constructing a brace from a pre-Lie algebra was described  on page 135. As suggested in \cite{Rump}, the addition in a given brace and in the corresponding pre-Lie algebra is the same, which makes this correspondence very simple. 
 In this paper, we provide a  formula for  passage from a brace of cardinality $p^{n}$ to a pre-Lie ring, 
 and  show that it works for  all braces under mild assumptions.
  For adjoint groups of strongly nilpotent braces, this provides us with  a very simple formula for obtaining a   Lie ring  isomorphic to the Lie ring obtained using the Lazard's correspondence, which  extends results from \cite{Lazard}. Recall that in \cite{Rump} a formula was developed for the passage from pre-Lie rings to braces using Lazard's correspondence. Intuitively, our formula for the passage from braces to pre-Lie rings  was first obtained for braces which were constructed  from pre-Lie rings using Lazard's correspondence, and the main difficulty was to present this formula in a simpler form and then  show that this formula will also hold  in general for all  strongly nilpotent braces.

 There are many open  questions concerned with which nilpotent groups are adjoint groups of braces. Nilpotent groups which are not adjoint groups of braces appear to be rare and difficult to construct  (see \cite{Rump, DB}). As mentioned in \cite{Rump}, this is related to Milnor's conjecture.
 It has been known since the 1950s (see \cite{Newman}) that the Baker-Campbell-Hausdorff formula gives an
isomorphism between the category of nilpotent Lie rings with order $p^{n}$ and the category
of finite $p$-groups with order $p^n$ provided that $p >n+1$. This
connection has been systematically exploited in order to classify finite $p$-groups in \cite{Newman}. It was suggested in   \cite{Rump}, that a similar approach can be used to classify adjoint groups of braces. Moreover, Bachiller showed that if $n+1<p$ then the additive and the multiplicative orders of elements in the underlying brace are the same \cite{DB}.
  Since the addition in the brace and in the corresponding Lie ring is the same, it simplifies calculations. The formula for the multiplication in the corresponding Lie ring is also very simple for adjoint groups of strongly nilpotent braces (see section 3).

We also show that passage from pre-Lie rings to braces works for all left nilpotent pre-Lie rings of cardinality $p^{n}$ where $n+1<p$.  
 This gives a $1$-to-$1$ correspondence between  the set of   strongly nilpotent braces  and the   set of  nilpotent pre-Lie rings of cardinality $p^{n}$ and nilpotency index $k$, for each $n+1<p, k<p$.

The following question remains unanswered however.
\begin{question}
Is there a passage from all left nilpotent braces of cardinality $p^{n}$, with $n+1<p$,    to left nilpotent pre-Lie rings?
 \end{question} 
 Our main result is as follows.

 \begin{theorem}\label{main111} Let $k,n$ be natural numbers and let $k+1, n+2\leq p$, where $p$ is a prime number. 
 There is a one-to-one correspondence between strongly nilpotent  braces of cardinality $p^{n}$ and nilpotency index $k$ (so $A^{[k]}= 0$ and $A^{[k-1]}\neq 0$)
 and nilpotent pre-Lie rings  of cardinality $p^{n}$ and nilpotency index $k$. 
\end{theorem}
 As mentioned by Vendramin in \cite{V}, braces and skew braces of cardinality $p^{n}$ can be viewed as building blocks of other braces. Indeed, it was shown in \cite{RN}
 that the multiplicative group of a brace and subbraces which are its Sylow's subgroups (which  have cardinality $p^{n}$) determine  immediately the whole brace.
 There is also a connection with {\em l}-groups \cite{CRZ}. 

Recall that a brace which is both left nilpotent and right nilpotent is strongly nilpotent \cite{Engel}. Right nilpotent braces have associated set-theoretic solutions of the Yang-Baxter equation of finite multipermutation level, and this class of set-theoretic solutions was investigated by several authors \cite{r, gateva, Cameron, GW}. The right nilpotency of braces and skew braces has also been investigated \cite{rump, CSV, Dora}. In the last section, we translate some results on right nilpotent braces to right nilpotent pre-Lie algebras.  At the moment, it seems that not right nilpotent braces and pre-Lie algebras  are less well understood (in the case when they are left nilpotent). It was not clear if the formula obtained in \cite{Rump} for the passage from pre-Lie rings to braces will hold for pre-Lie rings which are not right nilpotent. Our next result shows that this is indeed the case. 
\begin{theorem}\label{e} Let $p$ be a prime number and let $n+1<p$ be a natural number.
 Let $(A, +, \cdot )$ be a left nilpotent pre-Lie ring of cardinality $p^{n}$. Define the operation $\circ $ on $A$ by \[a\circ b=a+e^{L_{\Omega (a)}}(b),\]
 so $(A, \circ )$ is the group of flows of the pre-Lie algebra $A$. 
 Then $(A, +, \circ )$ is a left brace which is left nilpotent. The brace $(A, +, \circ )$ is right nilpotent if and only if   $(A, +, \cdot )$ is  a right nilpotent pre-Lie ring.
 Moreover, two distinct pre-Lie rings $(A, +, \cdot _{1})$ and $(A, +, \cdot _{2})$ give distinct braces.
\end{theorem}
  This answers Question $1$ from \cite{Lazard}.
 Theorem \ref{e}  prompts the following question:
\begin{question} Let $p$ be a prime number and let  $n<p-1$ be a natural number.
 Which  Lie rings of cardinality $p^{n}$ admit left nilpotent pre-Lie algebras? 
\end{question}
 In \cite{Rio}, Jespers, Ced{\' o} and Rio asked  which groups are adjoint groups of braces \cite{Rio}.
 Notice that by using  reasoning from page 141 in \cite{Rump}  and by using assumptions from Proposition \ref{e} we see that  if a Lie algebra of cardinality $p^{n}$ with $n+1<p$ admits a left nilpotent  pre-Lie algebra then the  Lazard's group of this Lie algebra is an adjoint group of a brace. See \cite{AG, M} for details of how the group of flows is related to the Lazard group.
Note that by the adjoint group of a brace $B$ we mean the group $(B, \circ )$ which is also called the multiplicative group of brace $B$. Notice that Lie algebras admitting pre-Lie algebras have been extensively investigated, and the survey paper \cite{Burde} contains many related results. However, it is not clear if similar methods as mentioned in \cite{Burde} can be used for finite Lie rings and which  pre-Lie algebras mentioned in \cite{Bai, Burde}  are left nilpotent.

\section{Background information}

 Recall that a   {\em pre-Lie ring} $A$ is  a vector space with a binary operation $(x, y) \rightarrow  x\cdot y$ and the operation $(x,y)\rightarrow x+y$ such that $(A, +)$ is an abelian group
satisfying
\[(x\cdot y)\cdot z -x\cdot (y\cdot z) = (y\cdot x)\cdot z - y\cdot (x\cdot z)\]
and 
\[(x+y)\cdot z=x\cdot z+y\cdot z, x\cdot (y+z)=x\cdot y+x\cdot z,\]
 for every $x,y,z\in A$. We say that a pre-Lie ring $A$  is {\em  nilpotent} or {\em strongly nilpotent}   if for some $n\in \mathbb N$ all products of $n$ elements in $A$ are zero. We say that $A$ is left nilpotent if for some $n$, we have $a_{1}\cdot (a_{2}\cdot( a_{3}\cdot (\cdots  a_{n})\cdots ))=0$ for all  $a_{1}, a_{2}, \ldots , a_{n}\in A$.

 Braces were introduced by Rump in 2007 \cite{rump} to describe all involutive, non-degenerate set-theoretic solutions of the Yang-Baxter equation. Braces quickly found applications in other research areas, such as Hopf-Galois extensions, matched pairs of groups, regular subgroups of a holomorph, groups with bijective 1-cocycles, factorised groups, Near rings, Jacobson radical rings, quantum integrable systems,  Garside monoids, {\em l}-groups and lattice theory, virtual knot theory, biracks  and trusses
 \cite{ DB, TB, Catino, Cedo, cjo, primitive, okninski, Rio,  FC, LC, pent, doikou, doikou2, Etingof, p, GW, ILau, LV, kayvan, Rump, Agatka1, SVB,  V}.
Recall that a set $A$ with binary operations $+$ and $* $ is a {\em  left brace} if $(A, +)$ is an abelian group and the following version of distributivity combined with associativity holds.
  \[(a+b+a*b)* c=a* c+b* c+a* (b* c), \space  a* (b+c)=a* b+a* c,\]
for all $a, b, c\in A$,  moreover  $(A, \circ )$ is a group, where we define $a\circ b=a+b+a* b$.
In what follows we will use the definition in terms of operation `$\circ $' presented in \cite{cjo} (see \cite{rump} for the original definition): a set $A$ with binary operations of addition $+$, and multiplication $\circ $ is a brace if $(A, +)$ is an abelian group, $(A, \circ )$ is a group and for every $a,b,c\in A$
\[a\circ (b+c)+a=a\circ b+a\circ c.\]
 
In \cite{rump}, Rump introduced {\em left nilpotent}  and  {\em right nilpotent}  braces and radical chains $A^{i+1}=A*A^{i}$ and $A^{(i+1)}=A^{(i)}*A$  for a left brace $A$, where  $A=A^{1}=A^{(1)}$. Recall that a left brace $A$  is left nilpotent if  there is a number $n$ such that $A^{n}=0$, where inductively $A^{i}$ consists of sums of elements $a*b$ with
$a\in A, b\in A^{i-1}$. A left brace $A$ is right nilpotent   if  there is a number $n$ such that $A^{(n)}=0$, where $A^{(i)}$ consists of sums of elements $a*b$ with
$a\in A^{(i-1)}, b\in A$.
 Strongly nilpotent braces and the chain of ideals $A^{[i]}$ of a brace $A$ were defined in \cite{Engel}.
 Define $A^{[1]}=A$ and \[A^{[i+1]}=\sum_{j=1}^{i}A^{[j]}*A^{[i+1-j]}.\]  A left brace $A$ is {\em strongly nilpotent}  if  there is a number $n$ such that $A^{[n]}=0$, where $A^{[i]}$ consists of sums of elements $a*b$ with
$a\in A^{[j]}, b\in A^{[i-j]}$ for all $0<j<i$.   Various other radicals in braces were subsequently introduced,  see \cite{JKVV, JespersLeandro}.
   We define left nilpotent, right nilpotent and strongly nilpotent pre-Lie rings in the same way, but  using the operation $\cdot $ instead of the brace operation $*$.
 All braces considered in this paper are left braces, but from now on we well just refer to them as braces.

\section{ Passage from finite braces to Pre-Lie rings}

By a result of Rump \cite{rump}, every brace of order $p^{n}$ is left nilpotent. Assume that $B$ is a brace which is both left nilpotent and  right nilpotent, then by a result from \cite{Engel} it is strongly nilpotent. In other words,  there is $k$ such that the product of any $k$ elements, in any order,  is zero (where all products are under the operation $*$).  If $B^{[k]}=0$ and $B^{[k-1]}\neq 0$, then we will say that $B$ is  strongly nilpotent of degree $k$, we also say that $k$ is the nilpotency index of $B$. 

We recall  Lemma 15 from \cite{Engel}:

\begin{lemma}\label{fajny}
 Let $s$ be a natural number and let $(A, +, \circ)$ be a left brace such that $A^{s}=0$ for some $s$.
 Let $a, b\in A$, and as usual define $a*b=a\circ b-a-b$.
Define inductively elements $d_{i}=d_{i}(a,b), d_{i}'=d_{i}'(a, b)$  as follows:
$d_{0}=a$, $d_{0}'=b$, and for $1\leq i$ define $d_{i+1}=d_{i}+d_{i}'$ and $d_{i+1}'=d_{i}d_{i}'$.
 Then for every $c\in A$ we have
\[(a+b)*c=a*c+b*c+\sum _{i=0}^{2s} (-1)^{i+1}((d_{i}*d_{i}')*c-d_{i}*(d_{i}'*c)).\]
\end{lemma}

{\bf Notation 1.} Let $A$ be a strongly nilpotent brace with operations $+, \circ , *$ defined as usual, so $x\circ y=x+y+x*y$
  for $x, y, z\in A$,  and 
let $E(x, y, z)\subseteq A$ denote the set consisting of any product of elements $x$ and $y$ and one element $z$ at the end of each product under the operation $*$,  in any order, with any distribution of brackets, each product consisting of at least 2 elements from the set $\{x,y\}$, each product having $x$ and $y$ appear at least once,  and having element $z$ at the end. Notice that $E(x,y,z )$ is  finite, provided that $A$ is a strongly nilpotent brace, as we may assume that all products of $k$ or more elements are $0$, where $k$ is the nilpotency index of $A$. Let $V_{x,y,z}$ be  a vector obtained from products of elements $x, y, z$ arranged in a such way that shorter products of elements  are situated  before longer products. We consider only products of less than $k$ elements. 

$ $

 Recall that  a number $\gamma $ is a {\em  primitive root}  modulo $m$ if every integer coprime to $m$ is congruent to a power of $\gamma $ modulo $m$.
 If $p>2$ is a prime number then it is known that there exists a primitive root modulo $p^{i}$ for every $i$. Let $k<p$ be the nilpotency index of a brace $A$, so $A^{[k]}=0$ and $A^{[k-1]}\neq 0$. Let $A$ have cardinality $p^{n}$ with $n+1<p$. 
 Let $\gamma $ be a primitive root modulo $p^{p}$ and $\phi $ be the Euler's function then $\phi (p^{p})=(p-1)p^{p-1}$. By Euler's theorem  $\gamma ^{p^{p-1}(p-1)} \equiv 1 \mod p^{p}$ and $\gamma ^{i}$ is not congruent to $1 \mod p^{p}$ for  natural $1\leq i<\phi (p^{p})=(p-1)p^{p-1}$. 

\begin{lemma}\label{3} Let $p>2$ be a prime number.
Let $\xi=\gamma ^{p^{p-1}}$ where $\gamma $ is a primitive root modulo $p^{p}$, then $\xi ^{p-1}\equiv 1 \mod p^{p}$.  Moreover,  $ \xi ^{j}$ is not congruent to $1$ modulo $p$ for natural $0<j<p-1$. 
\end{lemma}
\begin{proof} Notice that $\gamma $ is also a primitive root modulo $p$, since it is a primitive root modulo $p^{p}$.  Notice that $p-1$ divides $p^{p-1}-1$, therefore $\gamma  ^{p^{p-1}-1}\equiv 1 \mod p$ (since $\gamma ^{p-1}\equiv 1 \mod p$). Therefore, $\xi=\gamma ^{p^{p-1}}\equiv \gamma \mod  p$. Consequently, $e^{j}\equiv \gamma ^{j} \mod p$ for all $j$. It follows that $\xi ^{j}$ is not congruent to $1$ modulo $p$ for every $0<j<p-1$, since $\gamma ^{j}$ is not congruent to $1$ modulo $p$ for such $j$ (since $\gamma $ is a primitive root modulo $p$).
\end{proof}

Below we associate to every strongly nilpotent  brace a pre-Lie ring which is also strongly nilpotent and which has the same additive group.
 
\begin{proposition}\label{12345} Let $p$ be a prime number, and $n,k$ be natural numbers such that $n+1<p$ and $k<p$.
Let $A$ be a strongly nilpotent brace of cardinality $p^{n}$, and  let $k$ be the nilpotency index of $A$ so $A^{[k]}=0$ and $A^{[k-1]}\neq 0$. Let $\xi =\gamma  ^{p^{p-1}}$ where $\gamma $ is a primitive root modulo $p^{p}$.
 As usual, the operations on $A$ are $+$, $\circ $  and $*$ where $a*b=a\circ b-a-b$. Define the binary operation $\cdot $ on $A$ as follows 
\[a\cdot b=\sum_{i=0}^{p-2}\xi ^{p-1-i}((\xi ^{i}a)* b),\]
    for $a, b\in A$, where $\xi ^{i}a$ denotes the sum of $\xi ^{i}$ copies of element $a$.
Then, \[(a+b)\cdot c=a\cdot c+b\cdot c\] for every $a, b, c\in A$. 
Moreover \[a\cdot (b+c)=a\cdot b+a\cdot c\] for every $a,b,c\in A$.
\end{proposition} 
\begin{proof} 
By the definition of a left brace, we immediately get that  $a\cdot ( b+ c)=a\cdot b+ a\cdot c$. We will show that 
$(a+b)\cdot c=a\cdot c+b\cdot c$ for $a, b, c\in A$. 
 Observe that    \[(a+b)\cdot c= \sum_{i=0}^{p-2}\xi ^{p-1-i}((\xi ^{i}a+\xi ^{i}b)* c).\]
  Lemma \ref{fajny} applied several times yields 
\[\xi ^{p-1-n}(\xi ^{n} a+{\xi ^{n}}b)* c=
 \xi ^{p-1-n}((\xi ^{n} a)*c)+\xi ^{p-1-n}(({\xi ^{n}}b)*c)+ \xi ^{p-1-n}C(n),\] 
 where $C(n)$ is a sum of some products (under operation $*$) of elements $\xi ^{n}a$ and $\xi ^{n}b$ and an element $c$ at the end, because $A$ is a strongly nilpotent brace. Moreover, each product has at last one occurrence of element  ${\xi ^{n}}a$ and also at last one occurrence of element  ${\xi ^{n}}b$ and an element $c$ at the end.

 To show that  $(a+b)\cdot c =a\cdot c+b\cdot c$  it suffices to prove that \[\sum_{i=0}^{p-2}\xi ^{p-1-n}C(n)=0.\]
  We may consider a vector $V_{\xi ^{i}a, \xi ^{i}b, c}$ obtained as in Notation $1$ from products of elements $\xi ^{i}a$, $\xi ^{i}b$, $c$. 

By Lemma \ref{fajny} (applied several times) every element from the set $E(\xi x, \xi y, z)$ can be written as a linear combination 
   of elements from $E(x, y, z)$, with  integer coefficients which do not  depend on $x, y$ and $z$. We can then  organize these coefficients in a matrix, which we will call $M=\{m_{i,j}\}$,  so that we obtain
\[MV_{x,y, z}=V_{\xi x, \xi y, z}.\]

Notice that  elements from $E(x,y, z)$ (and from $E(\xi x, \xi y, z)$) which are shorter appear before elements which are longer in our vectors $V_{x,y, z}$ and $V_{\xi x, \xi y, z}$. Therefore by Lemma \ref{fajny} it follows that $M$ is an  upper triangular matrix.

%

Notice that $M$ does not depend on $x, y$ and $z$, as we only used relations from Lemma \ref{fajny} to construct it (and the fact that products of $k$ or more elements in $A$ are $0$). 
It follows that for every $n$, \[V_{\xi ^{n}x, \xi ^{n}y, z}=M^{n}V_{x,y, z}.\]
Observe that for $x,y\in A$ we have $\xi ^{p-1}x=x$ and $\xi ^{p-1}y=y$ because  $\xi ^{p-1}$ is congruent to $1$ modulo $p^{n}$ and $p^{n}x=p^{n}y=0$ since the group $(A, +)$ has cardinality $p^{n}$. 

 Therefore,  
\[V_{x, y, z}=V_{\xi ^{p-1}x, \xi ^{p-1}y, z}=M^{p-1}V_{x, y, z}=(\xi ^{(p-1)-1}M)^{p-1}V_{x, y, z}.\]

Notice that there is a vector $V$ with integer entries such that 
\[ C(n)=V^{T}V_{\xi ^{n}a, \xi ^{n}b, c}=V^{T}M^{n}V_{a, b, c}\] for each $n$, where $V^{T}$ is the transposition of $V$. 

 Denote $M^{0}=I$ the identity matrix.  Recall that $\xi ^{p-1}\equiv 1 \mod p^{n}$, since $n<p$, hence 

\[\sum_{n=0}^{p-2}\xi ^{(p-1)-n}C(n)=\sum_{n=0}^{p-2}(\xi ^{(p-1)-1})^{n}C(n)=V^{T}\sum_{n=0}^{p-2}(\xi ^{(p-1)-1}M)^{n}V_{a,b,c}.\]

Denote  \[Q_{a,b,c}=\sum_{n=0}^{p-2}(\xi ^{(p-1)-1}M)^{n}V_{a,b,c},\] then $\sum_{n=0}^{p-2}\xi ^{(p-1)-n}C(n)=V^{T}Q_{a, b, c}$. We need to show that 
 $\sum_{n=0}^{p-2}(\xi ^{(p-1)-1})^{n}C(n)=0$. It suffices to 
show that all entries of the vector $Q_{a, b, c}$ are zero.
Notice that 
\[(I-\xi ^{(p-1)-1}M)\sum_{n=0}^{p-2} (\xi ^{(p-1)-1}M)^{n}=I-(\xi ^{(p-1)-1}M)^{p-1},\] 
 therefore 
\[(I-\xi ^{(p-1)-1}M)Q_{a, b, c}=(I-(\xi ^{(p-1)-1}M)^{p-1})V_{a, b, c}=0,\]
 as shown before.
Therefore, entries of the vector $(I-\xi ^{(p-1)-1}M)Q_{a,b,c}$ are zero. Recall that entries of the vector $Q_{a, b, c}$ are elements of $A$ and therefore their additive orders are powers of $p$. 
 Recall that  that the diagonal entries of the matrix $I-\xi ^{(p-1)-1}M$ are coprime with $p$, and that this matrix is upper triangular. Therefore, $(I-\xi ^{(p-1)-1}M)Q_{a, b, c}=0$ implies $Q_{a, b, c}=0$, as required. 
\end{proof}

 We will now prove the main result of this section.

\begin{theorem}\label{main} Let $p>2$ be a prime number. 
  Let $A$ be a strongly nilpotent brace with nilpotency index $k$ and cardinality $p^{n}$ for some prime number $p$, and some natural numbers $k,n$ such that $k, n+1<p$. Define the binary operation $\cdot $ on $A$ as follows 
\[a\cdot b=\sum_{i=0}^{p-2}\xi  ^{p-1-i}((\xi ^{i}a)* b),\]
    for $a, b\in A$, where $\xi ^{i}a$ denotes the sum of $\xi ^{i}$ copies of element $a$.
Then \[(a\cdot b)\cdot c-a\cdot (b \cdot c)=(b\cdot a)\cdot c-(b\cdot a) \cdot c     \] for every $a, b, c\in A$. 
\end{theorem}
\begin{proof}  By Lemma \ref{fajny} applied several times we get 
\[(x+y)*z=x*z+y*z +x*(y*z)-(x*y)*z +d(x,y,z),\]
\[ (y+x)*z=x*z+y*z+y*(x*z)-(y*x)*z +d(y,x, z),\]
where $d(x, y, z)=E^{T}V_{x, y, z}$ for some vector $E$ with integer entries which does not depend of $x, y, z$, and where $V_{x, y, z}$ is as in Notation $1$ (moreover $d(x, y, z)$ is a combination of elements with at least 3 occurences of elements from the set $\{x, y\}$). 
 It follows that 
\[x*(y*z)-(x*y)*z-y*(x*z)+(y*x)*z=d(y, x, z)-d(x, y, z).\]

Let $a,b,c\in A$ and let $i, j$ be natural numbers.
  Applying it to $x=\xi ^{i}a$, $y=\xi ^{j}b$, $z=c$ we get 
\[(\xi ^{i}a)*((\xi ^{j}b)*c)-((\xi ^{i}a)*(\xi ^{j}b))*c+d(\xi ^{i}a, \xi ^{j}b,c)=\]
\[ =(\xi ^{j}b)*((\xi ^{i}a)*c)-((\xi ^{j}b)*(\xi ^{i}a))*c+d(\xi ^{j}b,  \xi ^{i}a, c).\]

Notice that  
\[a\cdot (b\cdot c)=a\cdot \sum_{j=0}^{p-2} \xi ^{p-1-j}((\xi ^{j}b)* c)=\sum_{i=0}^{p-2}\xi ^{p-1-i}((\xi ^{i}a)* 
\sum_{j=0}^{p-2}\xi ^{p-1-j}((\xi ^{j}b)* c)).\]

 Consequently, 
\[a\cdot (b\cdot c)=\sum_{i, j=0}^{p-2}\xi ^{p-1-i+p-1-j}((\xi ^{i}a)*((\xi ^{j}b)* c)).\]

On the other hand 

\[(a\cdot b)\cdot c=(\sum_{i=0}^{p-2}\xi ^{p-1-i}((\xi ^{i}a)* b))\cdot c=\sum_{i=0}^{p-2}\xi ^{p-1-i}(((\xi ^{i}a)* b)\cdot c),\]
 where the last equation follows from Proposition \ref{12345}. 

Consequently, 

\[(a\cdot b)\cdot c=\sum_{i,j=0}^{p-2}\xi ^{p-1-i+p-1-j}((\xi ^{j}((\xi ^{i}a)* b))* c)=\sum_{i,j=0}^{p-2}\xi ^{p-1-i+p-1-j}((\xi ^{i}a)* (\xi ^{j}b))* c,\]

 Recall a previous equation, multiplied by $e^{p-1-i+p-1-j}$ on both sides: 

\[\xi ^{p-1-i+p-1-j}(\xi ^{i}a)*((\xi ^{j}b)*c)-\xi ^{p-1-i+p-1-j}((\xi ^{i}a)*(\xi ^{j}b))*c+\xi ^{p-1-i+p-1-j}d(\xi ^{i}a, \xi ^{j}b,c)=\]
\[ =\xi ^{p-1-i+p-1-j}(\xi ^{j}b)*((\xi ^{i}a)*c)-\xi ^{p-1-i+p-1-j}((\xi ^{j}b)*(\xi ^{i}a))*c+\xi ^{p-1-i+p-1-j}d(\xi ^{j}b,  \xi ^{i}a, c).\]

 By summing the above equation for all $0\leq i,j\leq p-2$ and subtracting the previous equations we obtain that

\[a\cdot (b\cdot c)- (a\cdot b)\cdot c+\sum_{i,j=0}^{p-2}\xi ^{p-1-i+p-1-j}d(\xi ^{i}a, \xi ^{j}b,c)=\]
\[b\cdot (a\cdot c)- (b\cdot a)\cdot c+\sum_{i,j=0}^{p-2}\xi ^{p-1-i+p-1-j}d( \xi ^{j}b, \xi ^{i}a, c).\]

 So it remains to show that $\sum_{i,j=0}^{p-2}\xi ^{p-1-i+p-1-j}d(\xi ^{i}a, \xi ^{j}b,c)=0$, for all $a, b, c\in A$ 
 (and hence $\sum_{i,j=0}^{p-2}\xi ^{p-1-i+p-1-j}d( \xi ^{j}b, \xi^{i}a, c)=0$).

{\em Proof that  $\sum_{i,j=0}^{p-2}\xi ^{p-1-i+p-1-j}d(\xi ^{i}a, \xi ^{j}b,c)=0$}.  

 Notice that $d(a,b,c)=w(a,b,c)+v(a,b,c)$ where $w(a,b,c)$ contains all the products of elements $a,b,c$ which appear as summands in  $d(a,b,c)$ and in which $a$ appears at least twice, and 
$v(a,b,c)$ is a sum of products which are summands in $d(a,b,c)$ and in which $a$ appears only once (and hence $b$ appears at least twice).
It suffices to show that $\sum_{i,j=0}^{p-2}\xi ^{p-1-i+p-1-j}w(\xi ^{i}a, \xi ^{j}b,c)=0$ and $\sum_{i,j=0}^{p-2}\xi ^{p-1-i+p-1-j}v(\xi ^{i}a, \xi ^{j}b,c)=0$.
 Observe that  it suffices to show that $\sum_{i=0}^{p-2}\xi ^{p-1-i}w(\xi ^{i}a, b',c)=0$
   and  $\sum_{j=0}^{p-2}\xi ^{p-1-j}v(a', \xi ^{j}b, c)=0$ for any $a, a',  b, b', c\in A$ (notice that $a', b'$ should not be confused with inverses of $a$ and $b$).

 We will first show that \[\sum_{i=0}^{p-2}\xi ^{p-1-i}w(\xi ^{i}a, b',c)=0.\]

Observe that there is a vector $W$ with integer entries such that $w(a,b',c)=W^{T}V_{a,b',c}'$ where $V_{a,b',c}'$ is a vector constructed as in Notation 1 but only including as entries these products from $E(a,b',c)$  in which $a$ appears at least twice.
Consequently, 
\[ w(\xi ^{n}a, b',c)=W^{T}V'_{\xi ^{n}a, b',c}=W^{T}M^{n}V'_{a, b',c}\] for each $n$, where $W^{T}$ is the transposition of $W$ where $M$ is some upper triangular  matrix with integer entries such that 
  $V'_{\xi a, b, c}=MV'_{a,b',c}$ and $V'_{\xi ^{i}a, b, c}=M^{i}V'_{a,b',c}$ for every $i$  (this can be shown that such matrix $M$ exist  similarly as  in Proposition \ref{12345}).

Observe that the diagonal entries of the matrix $I-\xi ^{(p-1)-1}M$ are coprime with $p$, by Lemma \ref{3}. Indeed,  
 diagonal entries of $M$ are $\xi ^{i}$ where $2\leq i<k<p$, because $a$ appears at least twice in each product which is an entry in $V'_{a, b', c}$  (and because the nilpotency index of $A$ is $k$, so all products of length  longer than $k-1$ are zero).
 Denote $M^{0}=I$ the identity matrix. 
Now we calculate

\[\sum_{n=0}^{p-2}\xi ^{(p-1)-n}w(\xi ^{n}a, b',c) =\sum_{n=0}^{p-2}(\xi ^{(p-1)-1})^{n}w(\xi ^{n}a, b',c) =W^{T}\sum_{n=0}^{p-2}(\xi ^{(p-1)-1}M)^{n}V'_{a,b',c},\]
 where $e^{0}=1$ and $M^{0}=I$ the identity matrix.  
Denote  \[Q_{a, b', c}=\sum_{n=0}^{p-2}(\xi ^{(p-1)-1}M)^{n}V'_{a, b', c},\] then $\sum_{n=0}^{p-2}(\xi ^{(p-1)-1})^{n}w(\xi ^{n}a, b',c)=W^{T}Q_{a, b', c}$. We need to show that 
 \[\sum_{n=0}^{p-2}(\xi ^{(p-1)-1})^{n}w(\xi ^{n}a, b',c)=0.\] It suffices to 
 show that all entries of the vector $Q_{a, b', c}$ are zero.
Notice that 
\[(I-\xi ^{(p-1)-1}M)\sum_{n=0}^{p-2} (\xi ^{(p-1)-1}M)^{n}=I-(\xi ^{(p-1)-1}M)^{p-1},\]
 therefore 
\[(I-\xi ^{(p-1)-1}M)Q_{a, b', c}=(I-(\xi ^{(p-1)-1}M)^{p-1})V_{a, b', c}=0.\]
Therefore, entries of the vector $(I-\xi ^{(p-1)-1}M)Q_{a, b', c}$ are zero. Recall that entries of the vector $Q_{a, b', c}$ are elements of $A$ and therefore their additive orders are powers of $p$. 
 Recall that  that the diagonal entries of the matrix $(I-\xi ^{(p-1)-1}M)$ are coprime with $p$. Therefore, $(I-\xi ^{(p-1)-1}M)Q_{a, b',c}=0$ implies $Q_{a, b',c}=0$, as required.

 The proof that 
 $\sum_{j=0}^{p-2}\xi ^{p-1-j}v(a', \xi ^{j}b,c)=0$ for all $a', b,c\in A$ is similar.
 Observe that there is a vector $W'$ with integer entries  such that $v(a',b,c)=W'^{T}V_{a', b,c}''$, where $V_{a', b, c}''$ is a vector constructed as in Notation 1 but only including as entries those products in which $b$ appears at least twice. 
 By applying Lemma \ref{fajny} several times,  there exists a matrix $\bar M$ such that 
$V_{a', \xi b,  c}''={\bar M}V_{a', b,c}''$ and ${\bar M}$ is upper triangular with diagonal entries $\xi ^{i}$ for $i\geq 2$ (because $b$ appears at least twice  in each product which is an entry in $V_{a',b,c}''$)  and for $i< k<p$. 
 Similarly as before, \[V_{a',b,c}''=V_{a', \xi ^{p-1}b, c}''={\bar M}^{p-1}V_{a', b,c}''={\xi ^{p-1}\bar M}^{p-1}V_{a', b,c}''.\] It follows
 that 
\[ \sum_{j=0}^{p-2}\xi ^{p-1-j}v(a', \xi ^{j}b,c)=\sum_{j=0}^{p-2}\xi ^{p-1-j}W'^{T}V_{a', \xi ^{j}b,c}''=\sum_{j=0}^{p-2}W'^{T}(\xi ^{p-1-1}{\bar M})^{j}V_{a',b,c}''.\]
 
Observe that  $(1-\xi ^{p-1-1}{\bar M})\sum_{j=0}^{p-2}(\xi ^{p-1-1}{\bar M})^{j}V_{a',b,c}''=0$. Recall that  the diagonal entries of the upper triangular  matrix $I-\xi ^{p-1-1}{\bar M}$ are coprime with $p$. Consequently, 
 $\sum_{j=0}^{p-2}(\xi ^{p-1-1}{\bar M})^{j}V_{a',b,c}''=0$ which implies 
 \[\sum_{j=0}^{p-2}\xi ^{p-1-j}v(a', \xi ^{j}b, c)=0,\] this concudes the proof. 
\end{proof} 

$ $

 Let $A$ be a set and let  $(A, +_{1}, \circ _{1})$ and $(A, +_{2}, \circ _{2})$ be two braces. We say that these braces are the same if for each $a,b\in A$ we have $a+_{1} b=a +_{2} b$ and $a\circ _{1} b=a\circ _{2}b $. Otherwise, we say that these two braces are distinct. Similarly,  
  we say that pre-Lie rings  $(A, +_{1}, \cdot _{1})$ and $(A, +_{2}, \cdot _{2})$ are the same if and only if  for each $a,b\in A$ we have $a+_{1} b=a +_{2} b$ and $a\cdot _{1} b=a\cdot _{2}b $.
\begin{proposition}\label{2} 
Let notation be as in Theorem \ref{main}, then $a\cdot b$ is a sum of products of elements $a$ and $b$ under operation $*$ with $a$ appearing at least once and $b$ appearing exactly once at the end of each product. Moreover this formula depends only of the nilpotency index of the brace $A$. Moreover, the obtained pre-Lie ring $(A, +, \cdot )$ has the same nilpotency index as the brace $A$. In particular pre-Lie rings obtained from distinct braces, using Theorem \ref{main}, are distinct as pre-Lie rings.
\end{proposition}
\begin{proof} The proof is similar to the proof of Theorem 8 from \cite{Lazard}. 
 Let $E_{a,b}\subseteq A$  denote the set consisting of any product of elements $a$ and one element $b$ at the end of each product under the operation $*$,  in any order, with any distribution of brackets, each product consisting of at least 2 elements $a$.
 By Lemma \ref{fajny}, 
\[a\cdot  b=(p-1)a* b +\sum_{w\in E_{a,b}}\alpha _{w}w\] where 
$\alpha _{w}\in \mathbb Z$ do not depend on $a,b$, but only on their arrangement in word $w$ as an element of $E_{a,b}$. Moreover, 
 each $w$ is a product of at least $3$ elements from the set $\{a, b\}$. Observe that coefficients $\alpha _{w}$ do not  depend on the brace $A$ as they were constructed using the formula from Lemma \ref{fajny} which holds in every strongly nilpotent brace. Therefore,  for any given $n$, the same formula will hold for all braces of nilpotency index $k$.
 Therefore $a*b=-(1+p+p^{2}+\ldots +p^{n-1})(a\cdot b-\sum_{w\in E_{a,b}}\alpha _{w}w)$, and now we can use this formula several times  to write every element from $E_{a,b}$ as a product of elements $a$ and $b$ under the operation $\cdot $ (it will work because any product of $i$ elements under the operation $\cdot $ will belong to $A^{[i]}$ by the formula $a\cdot  b=(p-1)a* b +\sum_{w\in E_{a,b}}\alpha _{w}w$). Notice also  that if $a\in A^{[i]}$ and $b\in A^{[j]}$ then 
$a\cdot b\in A^{[i+j]}$.
 In this way we can recover the brace $(A,  +, \circ )$ from  the pre-Lie ring $(A, +, \cdot )$. 
Notice that this implies that  the nilpotency index of the obtained pre-Lie ring is the same as the nilpotency index of the brace $A$.
Therefore, two distinct strongly nilpotent braces cannot give the same pre-Lie ring using the formula $a\cdot b=\sum_{i=0}^{p-2}\xi ^{p-1-i}( (\xi ^{i}a)*b)$. 

\end{proof}

\begin{corollary}
 Let $(A, +, \cdot)$  be a pre-Lie ring and let $\alpha $ be a natural number.  Define $a\odot b=\alpha ( a\cdot b)$, then $(A, +, \odot)$ is a pre Lie ring.
\end{corollary}
\begin{proof} We see that $(a\odot b)\odot c-a\odot (b\odot c)=\alpha ^{2}((a\cdot b)\cdot c-a\cdot (b\cdot c))$, this implies that $A$ is a pre-Lie ring for every $a,b,c\in A$, this concludes the proof.
\end{proof}

{\bf Notation 2.}   Let $(A, +, \circ)$ be a strongly nilpotent brace of cardinality $p^{n}$ for some prime number $p\geq n+2$. Let  $(A, +, \cdot)$ be the pre-Lie ring obtained from this brace as in Theorem \ref{main}, and let $(A, +, \odot )$ be the pre-Lie ring obtained from $(A, +, \cdot )$ by declaring $a\odot b=-(1+p+p^{2}+\ldots +p^{n-1})(a\cdot b)$. 
 Then we will say that $(A, +, \odot)$ is the pre-Lie ring associated to the brace $(A, +, \circ )$.

\section{From pre-Lie rings to braces}

The passage from pre-Lie rings  to braces is the same as in Rump's paper \cite{Rump} on page 135. In  \cite{Lazard} this method was applied  for pre-Lie algebras over $\mathbb F_{p}$   using the group of flows. For pre-Lie rings the construction is the same and we recall it below.
 In the following section we  give a detailed explanation as to why it also holds  for finite left nilpotent pre-Lie rings.
  Recall that the group of flows was introduced in  \cite{AG}.

 Let $A$ with operations $+$ and $\cdot $ be a nilpotent  pre-Lie ring of nilpotency index $k$ and cardinality $p^{n}$. Recall  that  a pre-Lie ring $A$ is nilpotent of nilpotency index $k$ if  the product of any $k$ elements is zero in this pre-Lie ring, and $k$ is minimal possible. 
 Let $p$ be a prime number larger than $k$ and  $n+1$. Define the brace $(A,+, \circ )$ with the same addition as in the pre-Lie ring $A$ and with the multiplication $\circ $ defined  as in the group of flows as follows. The below exposition  uses the same notation as in \cite{Lazard}.

\begin{enumerate}

\item Let $a\in A$, and let  $L_{a}:A\rightarrow A$ denote the left multiplication by $a$, so
$L_{a}(b)=a\cdot b$.
 Define $L_{c}\cdot L_{b}(a)=L_{c}(L_{b}(a))=c\cdot (b\cdot a)$.
 Define \[e^{L_{a}}(b)=b+a\cdot b+{\frac 1{2!}}a\cdot (a\cdot b)+{\frac 1{3!}}a\cdot (a\cdot (a\cdot b))+\cdots \]
 where the sum `stops' at place $k$, since the nilpotency index of $A$ is $k$. This is well defined since $p>{k}$.

\item  We can formally add element $1$ to $A$, so $1\cdot a=a\cdot 1=a$ in our pre-Lie ring (see the next section)  and
 define \[W(a)=e^{L_{a}}(1)-1=a+{\frac 1{2!}}a\cdot a+{\frac 1{3!}}a\cdot (a\cdot a)+ \cdots \]
 where the sum `stops' at  place $k$. 
 Notice that $W:A\rightarrow A$ is a bijective function, provided that $A$ is a nilpotent pre-Lie ring.

\item Let $\Omega :A\rightarrow A$ be the inverse function to the function $W$, so  $\Omega (W(a))=W(\Omega (a))=a$.
  Following \cite{M} the first terms of $\Omega $ are
\[ \Omega (a)=  a-{\frac 12}a\cdot a +{\frac 14} (a\cdot a)\cdot a +{\frac {1}{12}}a\cdot (a\cdot a) +\ldots \]
 where the sum stops at place $k$. 
In \cite{M} the formula for $\Omega $ is given using Bernoulli numbers. This assures that $p$ does not appear in a denominator.
\item Define\[a\circ b=a+e^{L_{\Omega (a)}}(b).\]
 Here the addition is the same as in the pre-Lie ring $A$.
It was shown in \cite{AG}  that $(A, \circ )$ is a group.  The same argument will work in our case, as $(W(a)\circ W(b))\circ W(c)=W(a)\circ (W(b)\circ W(c))$ for $a,b,c\in A$ by BCH formula (at this stage the result is related to Lazard's correspondence).  We can immediately  see that $(A, +, \circ )$ is a left brace because
 \[a\circ (b+c)+a=a+e^{L_{\Omega (a)}}(b+c)+a=(a+e^{L_{\Omega (a)}}(b))+(a+e^{L_{\Omega (a)}}(c))=a\circ b+a\circ c.\]
\end{enumerate}
\section{The case of left nilpotent but not necessarily right nilpotent pre-Lie rings}

In this section we give additional explanations as to why the formula from the previous section also works for left nilpotent pre-Lie rings. 
 Notice that if $A$ is a left nilpotent pre-Lie ring of cardinality $p^{n}$ (where $p$ is prime) then $A^{n+1}=0$ since $A^{i}=A^{i+1}=A\cdot A^{i}$ implies $A^{i}=A^{i+1}=A^{i+2}=\ldots =0$. 
 Therefore, the formula for function $W$ from the previous section 
 is well defined. We give below detailed explanation why $W$ is also well defined for pre-Lie rings which are not pre-Lie algebras, for the reader's convenience.

 It is known that an identity element can be added to a pre-Lie algebra. Notice that the well known  construction of adding the identity element to an associative ring will also work for pre-Lie rings.

{\bf The Dorroh extension.} Each pre-Lie ring $A$  can be embedded in a ring with identity. In fact, if $
A$ is an arbitrary pre-Lie ring, then $(A\times \mathbb Z, +, \cdot )$
 together with 
\[(r,n)+(s,m)=(r+s, n+m)\]
\[(r,n)\cdot (s, m)=(rs+ns+mr, nm)\]
 for $r,s\in A$ and $n, m\in \mathbb Z$ is a ring with the identity element $(0,1)$. Notice that $rn$ denotes the sum of $n$ copies of $r$.  The identity element is $(0, 1)$ and the embedding is $r\rightarrow (r,0)$.
 This can be verified by a direct calculation. 
 This pre-Lie ring will be denoted as $A^{identity}$.
 
{\bf Notation 3.}  Let $A$ be a pre-Lie ring $A$ and $A^{identity}$ be  its Dooroh extension. By $S_{A}$ we denote the associative ring whose elements are  maps from $A^{identity}$ to $A^{identity}$ and such that if $\alpha \in S_{A}$ then $\alpha (a+b)=\alpha (a)+\alpha (b)$ for $a,b\in A^{identity}$.  In this ring we have the following operations for $\alpha:A^{identity}\rightarrow A^{identity} , \beta: A^{identity}\rightarrow A^{identity} \in S_{A}$ and $x\in A^{identity}$:
\[(\alpha +\beta )(x)=\alpha (x)+\beta (x), (\alpha \circ  \beta)(x)=\alpha (\beta (x)).\]

Let $a\in A$. We denote by $L_{a}:A^{identity}\rightarrow A^{identity}$ the left multiplication by $a$, so 
 $L_{a}(a)=a\cdot b$ where $\cdot $ is the multilpication in the pre-Lie ring $A$. Then $L_{a}$ is a nilpotent map, and  it belongs  to the associative ring $S_{A}$. Let $S_{A}'$ be the subring of $S_{A}$ generated by maps $L_{a}$ for $a\in A$. Note that $S_{A}'$ is a nilpotent ring
 and any product of $n+1$ elements in this ring is zero, provided that $A$ has cardinality $p^{n}$ and $A$ is left nilpotent.

{\em Remark.} It is known (see for example \cite{Newman}, page 390 or \cite{K}) that for any associative nilpotent ring $R$  with $R^{n+1}=0$ for some $n+1<p$, and for every $x,y\in R$  
 we have $e^{x}e^{y}=e^{v}$ where $v$ is obtained by the BCH formula.

\[v=x+y-{\frac 12}[x,y]+{\frac {1}{12}}[y,x,x]-{\frac {1}{12}}[y,x,y]+\cdots,\]
 where $[y,x,x]=[[y,x],x]$ and $[y,x,y]=[[y,x], y]$.
 We will use notation $v=BCH(x,y)$.

 Take $R=S_{A}'$, $x=L_{a}, y=L_{b}$ as above. Then we obtain the following associative multiplication

\[e^{L_{a}}\circ e^{L_{b}}=e^{BCH(L_{a}, L_{b})}.\]

Notice that \[BCH(L_{a}, L_{b})=L_{BCH(a,b)}\]
 because $A$ is a pre-Lie ring.

It follows from the fact that $L_{a}+L_{b}=L_{a+b}$ and 
$[L_{a}, L_{b}](c)=L_{a}(L_{b}(c))-L_{b}(L_{a}(c))=L_{[a,b]}(c)$ for $a,b\in A$, $c\in A^{identity}$.
 Indeed, notice that
 $[L_{a}, L_{b}](c)=a\cdot (b\cdot c)-b\cdot(a\cdot c)$ 
 and 
$L_{[a,b]}(c)=(a\cdot b-b\cdot a)\cdot c$,
which agrees for pre-Lie rings.

$ $

{\em Remark.} (This is a known fact, see Lemma 3.43 \cite{Burde} for pre-Lie algebras over real numbers.)  If a pre-Lie ring is left nilpotent then the associated Lie ring is nilpotent, where 
the multiplication in the associated Lie ring $L(A)$ is $[a,b]=a\cdot b-b\cdot a$ where $\cdot $ is the multiplication in the pre-Lie ring (and the addition is the same).  

$ $

Sketch of the proof: Let $A$ be a left nilpotent pre-Lie ring. Recall that we showed before that  $L_{a}L_{b}-L_{b}L_{a}=L_{[a,b]}$ for $a,b\in A$, since $A$ is a pre-Lie ring.
 Consequently, 
 \[[[L_{a},L_{b}], L_{c}]=[L_{[a,b]}, L_{c}]=L_{[[a,b], c]}.\]

 Continuing on in this way, we  obtain that  $L_{l}$
 for any $l$ in the Lie ring $L(A)$, can be obtained by applying opertors $L_{a}$ for $a\in A$ from the left side, so it will always reach zero. This follows because $[x,y]\cdot z=x\cdot (y\cdot z)-y\cdot (x\cdot z)$ in a pre-Lie ring. We can then show by induction that $[a_{1}, \ldots , a_{i}]\cdot z\in A^{i+k}$ for $a_{1}, \ldots , a_{n}\in A$ and $z\in A^{k}$. Notice that, $[[a_{1}, \ldots a_{i+1}]=[[a_{1}, \ldots , a_{i}], a_{i+1}]\cdot z=[a_{1}, \ldots , a_{i}]\cdot (a_{i+1}\cdot z)-a_{i+1}\cdot ([a_{1}, \ldots , a_{i}]\cdot z)$, and the result follows by applying  by induction with  respect to $i$. We can apply it for $z=1$. This shows that $L(A)$ is a nilpotent Lie ring.

\begin{lemma}\label{8}
 Let $(A, +, \cdot )$ be a left nilpotent pre-Lie ring  of cardinality $p^{n}$ for some prime number $p\geq n+2$. 
   Define $W(a)=e^{L_{a}}(1)-1$ for $a\in A$, $1\in A^{identity}$. Then $W(a)=W(b)$ for some $a, b\in A$ implies $a=b$. Therefore, $W:A\rightarrow A$ is a bijective map.
 Moreover every element $a\in A$ can be presented as $a=e^{L_b}(1)-1$ for some $b\in A$.  
Consequently, $(A, \circ )$ is a group, where we define
\[(e^{L_a}(1)-1) \circ (e^{L_{b}}(1)-1)=e^{L_{BCH(a,b)}}(1)-1,\] for $a, b\in A$.
\end{lemma} 
\begin{proof} Suppose $W(a)=W(b)$ for some $a,b\in A$, this implies $e^{L_a}(1)=e^{L_{b}}(1)$.
 By applying the map $e^{-L_{a}}$ to both sides we get $1=e^{L_{BCH(-a,b)}
}(1)$. Denote $s=L_{BCH(-a,b)}$ then $s\in S_{A}'$ and $s(1)+{\frac 12}s\cdot s(1)+\ldots =0$ (since $1=e^{s}(1)$ as above).
So $s(1)=f(s)(1)$ where $f(s)=-({\frac 12}s\cdot s+{\frac 16}s\cdot (s\cdot s)+ \cdots )$. Substituting (several times)  $f(s)(1)$ for $s(1)$ in the $s(1)$ at the end in each product appearing in the right side of the  equation $s(1)=f(s)(1)$ 
  we get that $s(1)\in A^{i}$ for every $i$, and since $A$ is a left nilpotent then $s(1)=0$.

Now, because $s(1)=0$ we get that 
$L_{BCH(-a,b)}=L_0$, hence 
 $BCH(L_{-a}, L_{b})=0$ (this follows because $A$ is a pre-Lie ring). This implies that 
$e^{L_{b}}$ is the inverse of $e^{L_{-a}}$ so 
$e^{L_{b}}=e^{L_{a}}$ in $S'_{A}$, so $e^{L_{b}}(x)=e^{L_{a}}(x)$ for all $x$ in $A^{identity}$.

We will now show that this implies that $a=b$.  Denote $c=b-a$. 
Note that  $A^{t}=0$ for some $t$ since $A$ is left nilpotent. Suppose that $c\neq 0$ and  let $i$ be such that $c\cdot A^{i}=0$.
 Let $x\in A^{i-1}$. Notice that 
$L_{c}L_{a}^{j}(x)\in c\cdot A^{i}=0$ for $j=1,2, \ldots $. Consequently,  
   \[e^{L_{a+c}}(x)=e^{L_{b}}(x)=e^{L_{a}}(x)\] for $x\in A^{identity}$  implies 
\[c\cdot x=L_{c}(x)\in \sum_{i=1,2, \ldots }{\mathbb Z}L_{a}^{j}L_{c}(x)=\sum_{i=1,2, \ldots }{\mathbb Z}L_{a}^{j}(c\cdot x).\]
 Now substituting this expression of $c\cdot x$ in the right hand side of the above expression we can increase the length of words appearing on the right hand side, hence we obtain $c\cdot x=0$, so $c=c\cdot 1=0$.

$ $

For the last part, notice that for $a,b\in A$ we have $L_{a}, L_{b}\in S_{A}'$
 hence it is known that
 \[e^{L_{a}}\circ e^{L_{b}}=e^{L_{BCH(a,b)}}.\]  We define  
\[(e^{L_a}(1)-1) \circ (e^{L_{b}}(1)-1)=e^{L_{BCH(a,b)}}(1)-1.\]

Note that knowing $e^{L_{a}}(1)-1$ we can recover $a $ and $L_{a}$ in the unique way, so this is a well defined multiplication. Indeed, note that by direct calculation we get that 
\[((e^{L_a}(1)-1) \circ (e^{L_{b}}(1)-1))\circ (e^{L_{c}}(1)-1)=(e^{L_a}(1)-1) \circ ((e^{L_{b}}(1)-1)\circ (e^{L_{c}}(1)-1)).\]
\end{proof}
  From Lemma \ref{8} we obtain: 
\begin{theorem}\label{99}
Let Notation be as in Lemma \ref{8}. Because $W:A\rightarrow A$ is a bijective function, then 
 there exists a function $\Omega :A\rightarrow A$ which is  the inverse function to the function $W$, so $\Omega (W(a))=W(\Omega (a))$ for all $a\in A$.
 Then $(A, +, \circ )$ is a brace, where $(A, \circ )$ is the group of flows of the pre-Lie ring $A$, so
\[a\circ b=a+e^{L_{\Omega (a)}}(b).\]
  Then $(A, +, \circ )$ is a left nilpotent brace.
\end{theorem}
\begin{proof}  Note that $(A, +, \circ )$ is a brace, since $(A, \circ )$ is a group, $(A, +)$ is an abelian group and  
\[a\circ (b+c)=a+e^{L_{\Omega (a)}}(b+c)= a\circ b+a\circ c-a.\]

 By a result of Rump from \cite{rump} every brace of cardinality $p^{n}$ is left nilpotent, hence $A$ is left nilpotent. 

\end{proof} 

\section{The  correspondence is one-to-one}

 In this section we consider strongly nilpotent braces and nilpotent nilpotent pre-Lie algebras, and we show that for them the correspondence is $1$-to-$1$.

We begin with the following lemma, which uses results from the previous section.

\begin{lemma}\label{x}
 Let $(A, +, \circ)$ be a left nilpotent pre-Lie ring  of cardinality $p^{n}$ for some prime $p>n+1$. 
 Let $\Omega:A\rightarrow A$ be  the inverse function of $W$ (so $W(\Omega (a))=a$) where $W(a)=e^{L_{a}}(1)-1$, where $1\in A^{identity}$.
  Then there are $\alpha _{w}\in \mathbb Z$ not depending on $a$ such that 
$\Omega (a)=a+\sum_{w }\alpha _{w}w(a)$ where $w$ are finite  non associative words in variable $x$ (of degree at least $2$), and $w(a)$ is a specialisation of $w$ at $a$ (so for example if $w=x\cdot (x\cdot x)$ then $w(a)=a\cdot (a\cdot a)$).
\end{lemma}
\begin{proof}
 Let $W(a)=e^{L_{a}}(1)-1$, and let $W^{2}(a)=W(W(a
))$ and $W^{3}(a)=W(W(W(a))) $, continuing in this way we define $W^{i}(a)$ for every $i$.
Note, that because $A$ is a finite pre-Lie ring it follows that for a given $a$ there are $0<i<j\leq p^{n}+1$ such that 
 \[W^{i}(a)=W^{j}(a).\]
 By Lemma \ref{8} this implies
 $W^{j-i}(a)=a,$ and consequently 
\[W^{p^{n}!}(a)=a,\] since $j-i$ divides $p^{n}!$ (since $i-j\leq p^{n}$).
 Notice that it holds for every $a\in A$. Therefore, $\Omega (a)=W^{p^{n}!-1}(a)$, since $W(\Omega (a))=\Omega (W(a))=W^{p^{n}!}(a)= a$. This concludes the proof.  
\end{proof}

 We are now ready to state our main results.

\begin{theorem}\label{y}
 Let $(A, +, \cdot )$ be a  nilpotent pre-Lie ring of cardinality $p^{n}$  and nilpotency index $k$ where $ k, n+1<p$, where $p$ is a prime number.
Let $(A, +, \circ )$ be the brace associated to $A$ by using the group of flows, as in previous sections (so $(A, \circ )$ is the group of flows of pre-Lie ring $(A, +, \cdot )$). 
 Let $(A, +, \odot)$ be the pre-Lie ring associated to the brace $(A, +, \circ )$ as in Notation $2$. 
 Then the pre-Lie rings $(A, +, \cdot )$ and $(A, +, \odot )$ are the same, so for each $a, b\in A$, $a\odot b=a\cdot b$.
\end{theorem}
\begin{proof} The proof is similar to the proof of Proposition 7  in \cite{Lazard}. We repeat it for readers' convenience. 
Notice that, since $(A, \circ )$ is the group of flows of the pre-Lie ring $A$, then  \[a* b=a\cdot b +\sum _{w\in P_{a,b}} \alpha _{w}w\] where $\alpha _{w}$ are integers   and $P_{a,b}$ is the set of all products (under the operation $\cdot $) of elements $a$ and $b$ from $(A, \cdot )$ with $b$ appearing only at the end, and $a$ appearing at least two times in each product. Moreover, $\alpha _{w}$ does not depend on $a$ and $b$, but only on their arrangement in word $w$ as an element of  set $P_{a,b}$.
 This follows from the construction of $\Omega (a)$, by Lemma \ref{x}. 
Notice that each word $w$ will be a product of at most $k-1$ elements because pre-Lie ring $A$ has nilpotency index $k$.
 Let $w\in P_{a,b}$, then $w$ is a product of some elements $a$ and element $b$. We define the word $w_{\xi ^{i}}$ to be the  word obtained if at each place where $a$ appears in $w$ we write $\xi ^{i}a$ instead of $a$.
 It follows that:
 \[(\xi ^{i}a)* b=(\xi ^{i}a)\cdot b +\sum _{w\in P_{a, b}} \alpha _{w}w_{\xi ^{i}}.\] 
 Consequently,
 \[\sum_{i=0}^{p-2}\xi ^{p-1-i}((\xi ^{i}a)* b)=\sum_{i=0}^{p-2}\xi ^{p-1-i}[(\xi ^{i}a)\cdot b+\sum _{w\in P_{a,b}} \alpha _{w}w_{\xi ^{i}}].\]

Notice that \[\sum_{i=0}^{p-2}\xi ^{p-1-i}(\xi ^{i}a)\cdot b=(p-1)a\cdot b.\]
 Therefore, \[-(1+p+p^{2}+\ldots +p^{n-1})\sum_{i=0}^{p-2}\xi ^{p-1-i}(\xi ^{i}a)\cdot b=(1-p^{n})a\cdot b=a\cdot b.\]
 Therefore, it suffices to show that for every $w\in P_{a,b}$ we have \[\sum_{i=0}^{p-2}\xi ^{p-1-i}w_{\xi ^{i}}=0.\] We know that every pre-Lie algebra is distributive, hence $w_{\xi ^{i}}=(\xi ^{i})^{j}w$ where $j$ is the number of occurences of $a$ in the product which gives $w$. 
 It suffices to show that $\sum_{i=0}^{p-2}\xi ^{p-1-i}(\xi ^{ij})\equiv 0 \mod p^{n}$  for $j\geq 2$. Because $2\leq j<k$ this is true, since $ (\xi ^{j-1}-1)\sum_{i=0}^{p-2}\xi ^{p-1-i}{\xi ^{ji}}= (\xi ^{j-1})^{p-1}-1\equiv 0 \mod p^{p}$, which concludes the proof.

  Notice that by the formula for the multiplication $*$ in the group of flows, the nilpotency index in the constructed brace (as the group of flows) will be the same as the nilpotency index of the pre-Lie algebra $A$. 
\end{proof}

We now show that the correspondence is one-to-one.
\begin{theorem}\label{13} 
 Let $(A, +, \circ  )$ be a strongly nilpotent brace of nilpotency index $k$ and cardinality $p$,  where $k, n+1<p$, where $p$ is a prime number.
 Let $(A, +, \odot )$ be a nilpotent pre-Lie ring obtained from this brace as in Notation 2.
 Then $(A, \circ )$ is the group of flows of the pre-Lie ring $(A, +, \odot )$.   
\end{theorem}
\begin{proof} The proof is similar the proof of  Theorem 8 in \cite{Lazard}.
 Recall that  $\xi ^{p-1}-1$ is divisible by $p^{n}$). 
Let $E_{a,b}\subseteq A$  denote the set consisting of any product of elements $a$ and one element $b$ at the end of each product under the operation $*$,  in any order, with any distribution of brackets, each product consisting of at least 2 elements $a$.
 Observe that by Lemma \ref{fajny} applied several times 
\[a\odot  b=a* b +\sum_{w\in E_{a,b}}\alpha _{w}w\] where 
$\alpha _{w}$ are integers which  do not depend on $a,b$, but only on their arrangement in word $w$ as an element of $E_{a,b}$. Moreover, 
 each $w$ is a product of at least $3$ elements from the set $\{a, b\}$. Observe that coefficients $\alpha _{w}$ do not  depend on the brace $A$ as they were constructed using the formula from Lemma \ref{fajny} which holds in every strongly nilpotent brace. Therefore,  for any given $n$, the same formula will hold for all braces of nilpotency index not exceeding $n$.
 Therefore $a*b=a\odot b-\sum_{w\in E_{a,b}}\alpha _{w}w$, and now we can use this formula several times  to write every element from $E_{a,b}$ as a product of elements $a$ and $b$ under the operation $\odot $. Note that $A^{[i]}\odot A^{[j]}\subseteq A^{[i+j]}$.
 In this way we can recover the brace $(A,  +, \circ )$ from  the pre-Lie algebra $(A, +, \odot )$. 

Therefore two distinct strongly nilpotent braces cannot give the same pre-Lie algebra using the formula from Notation 2. 

 Let $(A', +, \circ ')$ be the brace obtained as in Theorem \ref{99} from pre-Lie algebra $(A, +, \odot )$. By Theorem \ref{y} pre-Lie algebra $(A, +, \odot )$ can be obtained as in Notation $2$ from the brace $(A', +, \circ ')$. It follows that 
 $(A,\circ)'$ is the group of flows of pre-Lie algebra $(A, +, \odot )$. 

Notice that, by applying construction from  Notation 2 to braces $(A, +, \circ )
$ and $(A, +, \circ ')$ we obtain pre-Lie algebra $(A, +, \odot )$. By the above braces 
 $(A, +, \circ )
$ and $(A, +, \circ ')$ are the same, so $(A, \circ )$ is the group of flows of pre-Lie algebra $(A, +, \odot )$. 
\end{proof} 

\begin{theorem} Let $k,n$ be natural numbers and let $k+1, n+2\leq p$, where $p$ is a prime number. 
 There is one-to-one correspondence between strongly nilpotent  braces of cardinality $p^{n}$ and nilpotency index $k$ (so $A^{[k]}\neq 0$ and $A^{[k+1]}=0$)
 and nilpotent pre-Lie rings  of cardinality $p^{n}$ and nilpotency index $k$. 
\end{theorem}
\begin{proof}
 The proof is similar to the proof in the case when $A$ is an $\mathbb F_{p}$-brace, as in \cite{Lazard}. 
 For every pre-Lie ring of nilpotency index $k$ we can attach the brace which is its group of flows and form a   pair.  Since the group of flows is uniquely defined,  every pre-Lie algebra will be in exactly one pair.  Moreover, every brace will be in some pair, by Theorem \ref{13}. 
 Observe that every brace will be in exactly one pair, as otherwise there would be two distinct pre-Lie algebras which give the same group of flows. However, by Theorem \ref{y} we can apply the formula from Notation $2$,
\[a\odot b=-(1+p+p^{2}+\ldots +p^{n-1})\sum_{i=0}^{p-2}\xi ^{p-1-i}((\xi ^{i}a)* b),\] to recover these pre-Lie algebras from this brace. Because the formula defines uniquely the underlying pre-Lie algebra  every brace is in at most one pair.   
\end{proof}

 Notice that the group of flows is related to Lazard's correspondence, since 
$W(a)\circ W(b)=W(BCH(a,b))$, therefore the above correspondence is related to Lazard's correspondence  and the obtained structures are isomorphic to the structures obtained using Lazard's correspondence.
 
\section{Some ideals in braces of cardinality $p^{n}$}
Let $A$ be a brace of cardinality $p^{n}$ with $p>n+1$ and $ann (p^{i})$ be the set of elements of additive order $p^{i}$ in this brace, $ann(p^{i})=\{a\in A: p^{i}a=0\}$ where $p^{i}a$ denotes the sum of $p^{i}$ copies of element $a$.

Let $A$ be a brace and $a\in A$, by $a^{\circ n}$ we mean the product of $n$ copies of $a$ in the multiplicative group of brace, so $a^{\circ n}=a\circ a\circ \cdots \circ a$.
 By $a^{n}$ we will mean the product of $n$ copies of $a$ under the operation $*$, so $a^{n}=a*a*\cdots *a$. Recall that in any brace $A$, we use the notation $a*b=a\circ b-a-b$. 
 Let $i$ be a natural number, by $A^{\circ ^{p^{i}}}$ we denote the subgroup of the group $(A, \circ )$ generated by elements $a^{\circ p^{i}}$ where $a\in A$.  

 Recall that $A^{1}=A$ and $A^{i+1}=A*A^{i}$. By Rump's result if $A$ is a brace of cardinality $p^{n}$, where $p$ is a prime number,  then $A^{n+1}=0$ \cite{Rump}.
\begin{proposition}\label{1} Let $i,n$ be natural numbers.
 Let $A$ be a brace of cardinality $p^{n}$ for some prime number $p>n+1$.
  Then,  $p^{i}A=\{p^{i}a:a\in A\}$ is an   ideal in $A$ for each $i$.
 Moreover \[A^{\circ p^{i}}=p^{i}A.\]
\end{proposition}
\begin{proof} We need to show that $(p^{i}c)*b\in p^{i}A$, for all $b, c \in A$, as all other properties 
follow from the definition of a brace, since $c*(p^{i}b)=c*(b+b+\cdots +b)=p^{i}(c*b)$. 

  By Lemma 14 from \cite{note} we have that $a^{\circ p^{i}}*b=\sum_{k=1}^{p^{i}}{p^{i}\choose k}e_{k}$ where $e_{1}=a*b$, $e_{2}=a*(a*b)$, and $e_{k+1}=a*e_{k}$ for each $k$.

By Rump's result $A^{n+1}=0$ since $A$ is a brace of cardinality $p^{n}$, hence $e_{n+1}=0, e_{n+2}=0$.
Recall that $p>n+1$, it follows that 
$a^{\circ  p^{i}}*b\in p^{j}A$, since ${p^{i}\choose k}$ is divisible by $p^{i}$ for $k\leq p-1$. 
 
Let $A^{\circ p^{j}}$ be the subgroup of the group $(A, \circ )$ which is generated by elements $a^{\circ p^{j}}$ for $a\in A$.
 By Lemma 14 from \cite{note}, we have  $a^{\circ p^{i}}=\sum_{k=1}^{p^{i}}{p^{i}\choose k}a_{k}$ where $a_{1}=a$, $a_{2}=a*a$, and $a_{k+1}=a*a_{k}$ for each $k$. Similarly as above, it follows that 
 $A^{\circ p^{j}}\subseteq p^{i}A$. 
 Notice that, also by the above, if $g\in A^{\circ p^{j}}$ then $g*b\subseteq p^{j}A$ (this folows from the fact that
 $\lambda _{a}(\lambda _{b}(c))=\lambda _{a\circ b}(c)$ for $a,b,c\in A$ where
 $\lambda _{a}(c)=a*c+c=a\circ c-a$). 

It suffices to show that $p^{i}A\subseteq A^{\circ p^{i}}$. If $i=n$ then $a^{\circ p^{i}}=0$, since the group $(A, \circ )$ has $p^{n}$ elements, so $p^{n}A=0$, so the result holds.

 We proceed by induction on $j=n-i$. If $j=0$ then the result holds by the above. Supose it holds for some $i$ and we will show it holds for $i+1$. So we need to show that
$p^{i}A\subseteq A^{\circ p^{j}}$ provided that $p^{i+1}A\subseteq A^{\circ p^{i+1}}$. 

Let $a\in A$, we will show that $p^{i}a\in A^{\circ p^{i}}$. Observe first that $(-a)^{\circ p^{i}}\circ (p^{i}a)= a_{2}$ for some $a_{2}\subseteq p^{i}A^{2}$. Similarly,  
$(-a_{2})^{p^{i}}\circ a_{2}=a_{3}$ for some $a_{3}\subseteq p^{i}A^{3}$.
 Continuing in this way $(-a_{n})^{p^{i}}\circ \cdots \circ (-a_{2})^{p^{i}}\circ  (p^{i}a)=a_{n+1}\in A^{n+1}=0$ (the last equation follows from Rump's paper \cite{rump}).
 Therefore, $ (p^{i}a)=((-a)^{-1})^{p^{i}}\circ ((-a_{2})^{-1})^{p^{i}}\circ \cdots \circ ((-a_{n})^{-1})^{p^{i}}\in A^{\circ p^{i}}$. 
\end{proof}

\begin{corollary}
  Let $i,n$ be natural numbers.
 Let $A$ be a brace of cardinality $p^{n}$ for some prime number $p>n+1$. Let $I$ be an ideal in $A$. 
  Then,  $p^{i}I=\{p^{i}a:a\in I\}$ is an   ideal in $A$ for each $i$.
\end{corollary}
\begin{proof} Notice that $I$ is a brace, since it is an ideal. 
 By Proposition \ref{1}, applied to brace $I$ we get \[I^{\circ p^{i}}=p^{i}I.\]
 Therefore, $(p^{i}I)*A=I^{\circ p^{i}}*A\subseteq p^{i}I$. The last inclusion follows from the fact that 
  $a^{\circ p^{i}}*b=\sum_{k=1}^{p^{i}}{p^{i}\choose k}e_{k}$ where $e_{1}=a*b$, $e_{2}=a*(a*b)$, and $e_{k+1}=a*e_{k}$ for each $k$ (By Lemma 14 from \cite{note}). Moreover $e_{n+i}=0$ for $i=1,2 , \ldots $ since $I^{n+1}=0$, so ${p^{i}\choose k}=p^{i}\cdot k^{-1}{{p^{i}-1}\choose {k-1}}$ is divisible by $p^{i}$ for $k\leq p-1$. 
 Notice also that $a*(p^{i}b)=p^{i}a*b$ by the brace relation $a*(b+c)=a*b+a*c$. 
\end{proof}
\begin{lemma}\label{2} Let $n, i$ be a natural numbers. 
 Let $A$ be a brace  of cardinality $p^{n}$ for some prime number $p>n+1$. Let $ann (p^{i})=\{a\in A: p^{i}a=0\}$.
  Then,  $ann(p^{i})$ is an  ideal in $A$.
\end{lemma}
\begin{proof}
 Notice that $a*(p^{i}b)=p^{i}(a*b)$. Therefore,  if $a\in  A$ and $b\in ann (p^{i})$ then $a*b\in ann (p^{i})$. 
   To show that $ann(p^{i})$ is an ideal it suffices to show that $a*b\in ann(p^{i})$ provided that $a\in ann(p^{i})$, $b\in A$. 
  Note that if $a\in ann(p^{i})$ then $a^{\circ p^{i}}=0$ by the formula from Lemma $14$ in \cite{note}, since  
$a^{p^{i}\circ }=\sum_{k=1}^{p-1}{p^{i}\choose k}a_{k}$ where $a_{1}=a$, $a_{2}=a*a$, and $a_{k+1}=a*a_{k}$ for each $i$.
 It follows because $a_{n+1}\in A^{n+1}=0$ by Rump's Theorem, and  ${p^{i}\choose k}$ is divisible by $p^{i}$ for $k\leq p-1$. 

Recall also the formula 
$a^{\circ p^{i}}*b=\sum_{k=1}^{p-1} {p^{i}\choose  k}e_{k}$ where $e_{1}=a*b$, $e_{2}=a*(a*b)$, and $e_{k+1}=a*e_{k}$ for each $k$. 
 Therefore, for $a\in ann(p^{i})$,  $a^{\circ p^{i}}*b=0*b=0$, hence \[p^{i}e_{1}=-\sum_{k=2}^{p-1}{p^{i}\choose k}e_{k}\] since $A^{n+1}=0$ and  $e_{i}\in A^{i}$.  Notice that $p^{i}e_{2}=a*(a*(p^{i} a*b))=a*(a*(p^{i}e_{1}))$. Similarly $p^{i}e_{k}=a*(a*(\cdots a*(p^{i}e_{1})\cdots ))$.  
 Note that ${p^{i}\choose k}$ is divisible by $p^{k}$ for $k\leq p-1$, hence we can substitute the expression for $pe_{1}$ in the expressions for $ {{p^{i}} \choose k}e_{k}$ in the right-hand side  of the above equation. 
 Applying it several times we get $pe_{1}\in A^{n+1}=0$. Therefore, $pa=0$ implies $p(a*b)=0$ for each $b\in A$, hence $ann(p)$ is an ideal in brace $A$. 
\end{proof}

\section{Applications of brace theory for pre-Lie algebras}
 In this section we translate some results from brace theory to pre-Lie algebras with almost identical proofs.  Whilst several colleagues at conferences have suggested that these results may be new  for pre-Lie algebras, we would not be suprised if they were  known as  `common knowledge' to experts working in pre-Lie algebras.

 Recall that a pre-Lie ring is left nilpotent if \[A\cdot( A\cdot (A\cdots A))=0\] for some number of copies of $A$. We denote $A\cdot( A\cdot (A\cdots A))$ with $n$ ocurences of $A$ by $A^{n}$. We denote $((A\cdots A)\cdot A)\cdot A$, with $n$ occurences of $A$ by $A^{(n)}$. A pre-Lie ring $A$ is right nilpotent if $A^{(n)}=0$ for some $n$. 

We say that an ideal $I$ in a pre-Lie ring $(A, +, \cdot )$ is solvable if $I_{n}=0$ for some $n$, where $I_{1}=I$ and $I_{i+1}=I_{i}\cdot I_{i}$ where $I_{i}\cdot I_{i}$ consists of  sums of elements $a\cdot b$ where $a,b\in I_{i}$. 

$ $

 Recall that if $I$ is an ideal in a pre-Lie ring $A$ then $I\cdot A$ is the additive subgroup of $A$ generated by elements $i\cdot a$ where $a\in A, i\in I$. 

$ $

  A result  of Rump in \cite{rump} assures that if $I$ is an ideal in a brace $A$ then $I*A$ is also an ideal in $A$, and that $A^{(n)}$ is an ideal in a left brace $A$. This can be generalised for pre-Lie rings as follows:

\begin{proposition}
 Let $(A, +, \cdot )$ be a  pre-Lie ring and $I$ be an ideal in $A$ then $I\cdot A$ is also an ideal in $A$. In particular $A^{(n)}$ is an ideal in $A$.
\end{proposition}
 
\begin{proof} Observe that $I\cdot A\subseteq I$ since $I$ is an ideal in $A$, and hence $(I\cdot A)\cdot A\subseteq I\cdot A$. To show that $I$ is an ideal in $A$ it suffices to show that 
$A\cdot (I\cdot A)\subseteq I\cdot A $. Let $a, c\in A$, $b\in I$ then we have by the pre-Lie ring relation:
\[(a\cdot  b)\cdot c- a\cdot (b\cdot c)=(b\cdot a)\cdot c- b\cdot (a\cdot c).\]
Therefore $a\cdot (b\cdot c)=(a\cdot b)\cdot c+b\cdot (a\cdot c)-(b\cdot a)\cdot c\subseteq I\cdot A$. 
It follows that $A\cdot (I\cdot A)\subseteq I\cdot A$. 
\end{proof}

 We can defined the socle of a pre-Lie ring $A$ as 
\[Soc (A)=\{ a\in A: a\cdot A=0\}.\]
 Rump \cite{Rump} defined a socle of a brace and showed that the socle of a brace is always an ideal in this brace.
 We show below that a similar result holds for pre-Lie rings.

 \begin{proposition} 
 Let $(A, +, \cdot )$ be a pre-Lie ring, then the socle of $A$ is an ideal in $A$.
\end{proposition}
\begin{proof}
 Let $a\in Soc(A)$ then $a\cdot b=0$ for $b\in A$, so it remains to show that $b\cdot a\in Soc(A)$. 
  It suffices to show that $(b\cdot a)\cdot c=0$ for $b, c\in A$. Notice that by the pre-Lie ring relation 
\[(a\cdot  b)\cdot c- a\cdot (b\cdot c)=(b\cdot a)\cdot c- b\cdot (a\cdot c).\]
 Hence, $(b\cdot a)\cdot c=0$ since $a\cdot A=0$.
\end{proof}

We say that a pre-Lie ring $(A, +, \cdot )$  is  nilpotent if $A^{[i]}=0$ for some $n$, where we define by induction
$A^{[1]}=A$ and $A^{[i+1]}=\sum_{j=1}^{i}A^{[j]}\cdot A^{[i+1-j]}$.
 Notice that $A^{[i]}$ is an ideal in $A$.

It was shown in \cite{Engel} that braces which are left nilpotent are strongly  nilpotent 
  if and only if they are right nilpotent.  Here we will show that an analogous result holds for pre-Lie rings.

\begin{theorem}\label{4}  Let $A$ be either a pre-Lie ring or a brace. Let $m, n$ be natural numbers, then there is $s_{n,m}$ depending only of $m$ and $n$ such that if   
 $A^{(n)}=0$ and $A^{m}=0$ then $A^{[s_{m,n}]}=0$.  Moreover, $s_{n+1, m}\leq s_{n,m}(m+1)$.
\end{theorem}
\begin{proof}  We use a similar proof as in Theorem 12  in  \cite{Engel}. If $A$ is a brace then this shows that such $s_{m,n}$ exists by Theorem 12 \cite{Engel}.

Suppose now that $A$ is a pre-Lie algebra. 
We will proceed by induction on $n$. If $n=0$ then $0=A^{2}=A\cdot A=A^{(2)}=A^{[2]}$, so the result holds.

Suppose that there is a natural number $s_{n,m}$ such that any left brace satysfying 
  $A^{(n)}=0$ and $A^{m}=0$ satisfies $A^{[s_{n,m}]}=0.$
 Assume now that our pre-Lie algebra satisfies $A^{(n+1)}=0$ and $A^{m}=0$.  Let $\alpha \geq s_{n,m}\cdot (m+1)$ and suppose that $a\in A^{[\alpha ]}$. 
 Then $a=\sum a_{i}\cdot b_{i}$ with $a_{i}\in A^{[q_{i}]}$, $b_{i}\in A^{[\alpha -q_{i}]}$. Observe that if $q_{i}\geq s_{n,m}$ then $a_{i}\in A^{(n)}$ 
(by the inductive assumption applied to pre-Lie algebra  $A/A^{(n)}$; this pre-Lie algebra is well defined since $A^{(n)}$ is an ideal in $A$).
 In this case we get $a_{i}\cdot b\in A^{(n)}\cdot A=A^{(n+1)}=0$. 
 Consequently we can assume that all $q_{i}< s_{n,m}$. 

 For each $i$ we can now write $b_{i}=\sum_{j}a_{i,j}\cdot b_{i,j}$, and by the same argument as before, we get that each $a_{i,j}\in A^{[r_{i,j}]}$ for some $r_{i,j}< s_{n,m}$ 
(as otherwise $a_{i,j}\in A^{(n)}$ by the inductive assumption applied to $A/A^{(n)}$, and so $a_{i,j}\cdot b_{i,j}\in A^{(n+1)}=0$).

Observe now that since $A$ is a left brace then 
\[ \sum_{i}a_{i}\cdot b_{i}=  \sum_{i}(a_{i}\cdot \sum_{j}a_{i,j}\cdot b_{i,j})=\sum_{i,j} a_{i}\cdot (a_{i,j}\cdot b_{i,j})  .\]

 Continuing in this way we get that $a\in \sum_{c_{1}, c_{2}, \ldots , c_{m}}(c_{1}\cdot (c_{2}\cdots c_{m-1}\cdot (c_{m}\cdot A))$ and since $A^{m}=0$ we get that each $a=0$, so $A^{[\alpha ]}=0$.
\end{proof} 
 Notice that if $A$ is a right nilpotent  pre-Lie ring (or a brace) of cardinality $p^{n}$ then $A^{(n+1)}=0$ since $A^{(i)}=A^{(i+1)}$ implies $A^{(i)}=A^{(i)}\cdot A=A^{(i+2)}=\cdots =0$. Similarly, if $A$ is left nilpotent then $A^{n+1}=0$ since $A^{i}=A^{i+1}$ implies $A^{i}=A^{i+1}+A^{i+2}+\ldots =0$. 
\begin{corollary}\label{12}
 Let $A$ be either a pre-Lie ring or a  brace of cardinality $p^{n}$ for some $n$. Suppose that $A$ is both left nilpotent and right nilpotent. Then 
 that $A^{[(n+1)^{n+1}]}=0$.
\end{corollary} 
 We are now ready to present proof of Theorem \ref{e}. 

$ $

{\bf Proof of Theorem \ref{e}.} The first part of Theorem $2$ follows from Theorem \ref{99}.
Let $(A, +, \circ )$ be a pre-Lie ring of cardinality $p^{n}$ which is left nilpotent, then $A^{n+1}=0$, hence $A^{p-1}=0$. Notice that $W(\Omega (a))=a$ for every $a\in A$, hence
 $a=\Omega (a)+ {\frac 12}\Omega (a) \cdot \Omega (a)+\cdots .$
Denote  $e_{1}=\Omega (a)$, $e_{2}=\Omega (a)\cdot \Omega (a)$ and $e_{i+1}=\Omega (a)\cdot e_{i}$ for each $i$.
 Consequently, 
\[a=e_{1}+ {\frac 12}e_{2}+{\frac 1{3!}}e_{3}+\cdots .\]
 Notice that by the formula for the multiplication in the group of flows we get 
\[a*a=e^{L_{\Omega (a)}}(a)-a=e_{2}+\sum_{i=2}^{n-1}\alpha _{i}e_{i},\]
 for some integers $\alpha _{i}$ not depending on $a$ and not depending on the multiplication $\cdot $ in the pre-Lie algebra $A$.
 Denote $f_{1}=a$, $f_{2}=a*a$, and $f_{i+1}=a*f_{i}$.  By calculating expresions for $f_{i}$ similarly as above, we obtain that
 \[f_{i}=\sum_{j=i}^{p-1}m_{j}e_{j} \] for some integers $m_{j}$ (not depending on $a$ and not depending on multiplication $\cdot $). 
 Consequently,  there are some integers  $\beta _{i}$
  (not depending on $a$ and not depending on  multiplication $\cdot $)
 such that \[e_{1}=\sum_{i=1}^{p-1}\beta _{i}f_{i}.\] 
  Recall that $e_{1}=\Omega (a)$.  Notice now that for element $b\in A$ we have 
\[a*b=e'_{1}+ {\frac 12}e'_{2}+{\frac 1{3!}}e'_{3}+\cdots \] where 
$e'_{1}=\Omega (a)\cdot b$ and $e'_{i+1}=\Omega (a)\cdot e'_{i}$. 
Reasoning similarly as above, we can calculate elements $f'_{i}$ in the brace $(A, +, \circ )$, where $f'_{1}=a*b$, $f'_{2}=a*(a*b)$ and 
$f'_{i+1}=a*f'_{i}$, and obtain that 
 \[\Omega (a)\cdot b =\sum_{i=1}^{p-1}\gamma _{i}f'_{i},\]
 for some integers $\gamma _{i}$ not depending on $a$ and $b$ and not depending on multiplication $\cdot $ in pre-Lie algebra $(A, +, \cdot )$.

We will now  prove uniqueness.
 Let $(A, +, \cdot  _{1})$ and $(A, +, \cdot _{2})$ be two left nilpotent pre-Lie rings  and let $(A, +, \circ _{1})$ and $(A, +, \circ _{2})$ be braces which are  their  groups of flows. 
 Suppose that braces $(A, +, \circ _{1})$ and $(A, +, \circ _{2})$ are identical, and denote them as $(A, +, \circ )$.
 By the formula, \[\Omega (a) =\sum_{i=1}^{p-1}\beta _{i}f_{i}\] the element 
 $\Omega (a)$ calculated in the first pre-Lie algebra using the  multiplication $\cdot _{1}$ gives the same element in the set $A$, as the element $\Omega (a)$ calculated 
 using  the multiplication $\cdot _{2}$. Moreover $\Omega (a)\cdot_{1} b=\Omega (a)\cdot _{2} b$ by the formula
 \[\Omega (a)\cdot b =\sum_{i=1}^{p-1}\gamma _{i}f'_{i}.\] 
Indeed, this formula implies,  \[\Omega (a)\cdot_{1} b =\sum_{i=1}^{p-1}\gamma _{i}f'_{i}=\Omega (a)\cdot_{2} b.\]
 By Lemma \ref{8}, 
$W:A\rightarrow A$ is a bijective function, hence $\Omega : A\rightarrow A$ is a bijective function, so every element in $A$ can be presented as $\Omega (a)$ for some $a\in A$.
  Therefore, pre-Lie algebras $(A, +, \cdot _{1})$ and $(A, +, \cdot _{2})$ are equal.

 Suppose now that $(A, +, \cdot )$ is right nilpotent, then it is  nilpotent by Corollary \ref{12}.  By Lemma \ref{x} and by the formula for the multiplication in the group of flows it follows that  $(A, +, \circ )$ is right nilpotent. 
 On the other hand, let  
$(A, +, \cdot )$ be not right nilpotent. Suppose on the contrary, that the obtained brace $(A, +, \circ )$ is right nilpotent, then it is strongly nilpotent by Corollary \ref{12}.
 By using the formula  $\Omega (a) =\sum_{i=1}^{p-1}\beta _{i}f_{i}$
 we see that $\Omega (a)\in A^{[i]}$ provided that $a\in A^{[i]}$ where $A^{[i]}$ is defined as usually for brace $(A, +, \circ )$.
 Now using the formula
 \[\Omega (a)\cdot b =\sum_{i=1}^{p-1}\gamma _{i}f'_{i},\]
  we see $A\cdot A\subseteq A*A$, and if $a\in A^{[i]}$ and $b\in A^{[j]}$ then $a\cdot b\in A^{[i+j]}$. Recall that $A^{[m]}=0$ for some $m$. This shows that pre-Lie algebra $(A, +, \cdot )$ is nilpotent.  
 
\begin{corollary}\label{galois}
 Let $n$ be a natural number and let $p>(n+1)^{n+1}$ be a prime number. Then every brace of cardinality $p^{n}$ which is right nilpotent is obtained as in Theorem $2$
 from some nilpotent pre-Lie ring. Moreover, braces which are groups of flows of given  pre-Lie rings are isomorphic if and only if these pre-Lie rings are isomorphic.
\end{corollary}

\section{Hopf-Galois extensions and pre-Lie rings}
 In \cite{DB} Bachiller discovered that there is a correspondence between braces and Hopf-Galois extensions of abelian type. 
 This was subsequently investigated in a series of papers \cite{BB, kayvan, SVB, LC, Koch} where a correspondence with skew braces was investigated, new structures were investigated such as bi-braces, and some details were clarified. 
  Let $(A, +)$ be an abelian group of cardinality $p^{n}$ for some $n$ and some prime number $p>(n+1)^{n+1}$.    An exposition from the last section from \cite{dsk} about correspondence between braces and Hopf-Galois extension, combined with Corollary \ref{galois} yields the following  method of of constructing  Hopf-Galois structures  of type $ (A,+) $.
 
$ $

{\bf  To construct Hopf-Galois structures  of type $ (A,+) $} we can proceed as follows:

	   \begin{enumerate}
	   \item  First we find all not right nilpotent  braces  with additive group $(A, +)$. We list these braces as  $A_{1}', \ldots A_{m}'$, where $A_{i}'=(A, +, \circ  _{i}')$.
\item We then  find all Hopf-Galois extensions related to  brace $A_{i}'$ for $i=1, 2, \ldots , m$, using the method described in the last section in \cite{dsk}.
	   
	   	  \item It remains to  construct all Hopf-Galois extensions related to right nilpotent  braces with additive group $(A, +)$. This can be done as follows:  
	  \item  We construct all nilpotent pre-Lie rings with abelian group $(A, +)$.  Recall that a pre-Lie ring is nilpotent if for some $j$ product of any $j$ elements in this pre-Lie ring is zero. 

	  \item We list all pairwise not isomorphic pre-Lie rings of cardinality $p^{n}$ as $A_{1}, A_{2}, \ldots  A_{k}$, where $A_{i}=(A, +, \cdot _{i})$ where $(A, +)$ for $i=1, 2, \ldots , k$. 
	  
	  \item  For a  pre-Lie ring $A_{i}$ we  construct all Hopf-Galois extensions related to this pre-Lie ring using the method described below.

\item For a given pre-Lie ring $A_{i}$ we obtain a Hopf-Galois extensions of type $(A, +)$ with Galois group isomorphic to the group of flows of the pre-Lie ring $A_{i}$. 
\end{enumerate}	  
	  
 We can construct all {\bf  Hopf-Galois extensions related to a pre-Lie ring   $(A, +, \cdot )$} as follows:
	\begin{enumerate}
		\item Let $(A, +,  \cdot)$ be a pre-Lie ring.
		\item Let $\gamma _{i}$ be the automorphisms of the abelian group $(A, +)$.
		\item Check which $\gamma _{i}$ are also automorphisms of  the whole pre-Lie ring $(A,+, \cdot)$.

 {\em  Remark. Notice that $\gamma $ is an automorphism of the pre-Lie ring $(A, +, \cdot)$ if and only if $\gamma $ is an automorphism 
 of the  the group of flows of $(A, +, \cdot)$, by Theorem \ref{13}.} 
		\item We find coset representatives of $\mathrm{Aut}(A,+)/\mathrm{Aut}(A,+,\cdot) $, we call them $\xi _{1}, \xi_{2}, \ldots $. 

Recall that automorphisms $\alpha $ and $\beta $ of $(A, +)$  are the same in $\mathrm{Aut}(A,+)/\mathrm{Aut}(A,+,\cdot) $ if and only if  $\alpha =\beta \gamma $ for some $ \gamma \in \mathrm{Aut}(A,+,\cdot ) $ (so $\alpha (a)=\beta (\gamma (a))$ for all $a\in A$).
		\item For the automorphisms $\xi _{i}$ of the group $(A, +)$ from our list above we find the pre-Lie ring
$(A, +, \cdot _{\xi_{i}})$, with the same addition as in $(A, +, \cdot )$  and with the multiplication \[a\cdot _{\xi_{i}}b=\xi_{i}^{-1}(\xi_{i}(a)\cdot  \xi_{i}(b)).\]

 {\em Remark. Notice that if $(A, +, \circ )$ is the group of flows of  the pre-Lie ring $(A, +, \cdot  )$ and $(A, +, \circ _{\xi _{i}})$ is the group of flows of  $(A, +, \cdot _{\xi _{i}})$ then $a\circ _{\xi _{i}} b= \xi_{i}^{-1}(\xi_{i}(a)\circ \xi_{i}(b))$.}
		\item The corresponding Hopf-Galois extension is the regular subgroup $(a, \lambda _{a})$ (denoted above as $a\lambda _{a}$) of the holomorph $\mathrm{Hol}(A, +)$ where
		 \[\lambda _{a}(b)=a\circ _{\xi _{i}}b-a,\] where $a\circ _{\xi _{i}}b$ is calculated using the formula for the group of flows of the pre-Lie ring $(A, +, \cdot _{\xi _{i}})$. 
		 
 We collect all Hopf-Galois extensions obtained in this way for all $\xi _{i}$ to obtain all Hopf-Galois extensions corresponding to the pre-Lie ring $(A, +, \cdot )$.  The obtained Hopf-Galois extensions are pairwise distinct.
	
\end{enumerate}


$ $

{\bf Acknowledgments.} The author acknowledges support from the
EPSRC programme grant EP/R034826/1 and from the EPSRC research grant EP/V008129/1. The author is grateful to the University of Edinburgh for providing her with a sabbatical in the Fall semester 2021.

\end{document}